\documentclass[12pt,a4paper]{article}
\usepackage{amsmath,amsthm,amsfonts,amssymb,bbm,color}
\usepackage{graphicx,psfrag,subfigure,url}
\usepackage{cite}
\usepackage{hyperref}
\usepackage[english]{babel}

\newcommand{\Id}{\mathbbm{1}}
\newcommand{\Or}{\mathcal{O}}
\newcommand{\E}{\mathbbm{E}}
\newcommand{\e}{\varepsilon}
\newcommand{\I}{{\rm i}}

\newcommand{\R}{\mathbb{R}}

\newcommand{\Z}{\mathbb{Z}}

\newcommand{\x}{\mathbf{x}}

\numberwithin{equation}{section}

\DeclareMathOperator{\Ai}{Ai}

\DeclareMathOperator{\Pb}{\mathbbm{P}}

\newtheorem{prop}{Proposition}[section]
\newtheorem{thm}[prop]{Theorem}
\newtheorem{lem}[prop]{Lemma}
\newtheorem{defin}[prop]{Definition}
\newtheorem{cor}[prop]{Corollary}

\newtheorem{rem}[prop]{Remark}
\newenvironment{remark}{\begin{rem}\normalfont}{\end{rem}}

\title{The second class particle process at shocks}

\author{Patrik L.\ Ferrari\thanks{Institute for Applied Mathematics, Bonn University, Endenicher Allee 60, 53115 Bonn, Germany. E-mail: {\tt ferrari@uni-bonn.de}} \and
Peter Nejjar\thanks{Institut für Mathematik, Universit\"at Potsdam, Karl-Liebknecht-Str.~24-25, 14476 Potsdam. E-mail: {\tt peter.nejjar@uni-potsdam.de}}
}

\date{December 19, 2023}

\begin{document}
\sloppy
\maketitle

\begin{abstract}
We consider the totally asymmetric simple exclusion process (TASEP) starting with a shock discontinuity at the origin, with asymptotic densities $\lambda$ to the left of the origin and $\rho$ to the right of it and $\lambda<\rho$. We find an exact identity for the distribution of a second class particle starting at the origin. Then we determine the limiting joint distributions of the second class particle. Bypassing the last passage percolation model, we work directly in TASEP, allowing  us to extend previous one-point distribution results via a more direct and shorter ansatz.
\end{abstract}

\section{Introduction}
The totally asymmetric simple exclusion process (TASEP) on $\Z$ is one of the simplest non-reversible interacting stochastic particle system. The occupation variables of the TASEP at time $t$ are denoted by $\eta_t(x)$ with $x\in\Z$, where $\eta_t(x)=1$ means site $x$ is occupied and $\eta_t(x)=0$ means site $x$ is empty (at time $t$). The TASEP dynamics is simple: particles jump one step to the right and are allowed to do so only if their right neighboring site is empty. Jumps are independent of each other and all have rate $1$.

More precisely, denoting by $\eta\in \Omega=\{0, 1\}^\Z$ a configuration, the infinitesimal generator of TASEP is given by the closure of the operator $L$ given by
\begin{equation}
Lf(\eta)=\sum_{x\in\Z}\eta(x)(1-\eta(x+1))(f(\eta^{x,x+1})-f(\eta)),
\end{equation}
where $f$ is a cylinder function (depending on finitely many occupation variables) and $\eta^{x,x+1}$ denotes the configuration obtained from $\eta$ by exchanging the occupation variables at $x$ and $x+1$. For the general theory and well-definiteness of the semigroup generated by $L$, see e.g.~\cite{Li85b}.

Under hydrodynamic scaling, the particle density $\varrho$ is well known to evolve according to the Burgers equation (see e.g.~\cite{BGRS10} for a much more general result)
\begin{equation}\label{eq1}
\partial_t \varrho+\partial_x \varrho(1-\varrho)=0.
\end{equation}
In particular, for the initial condition $\varrho(x,0)=\lambda \Id_{x\leq 0}+\rho \Id_{x> 0}$, with $\lambda<\rho$, the unique entropy solution of \eqref{eq1} (see e.g.~Theorem 4 of in \S 3.4.4 of~\cite{Ev10} for entropy solutions) is
\begin{equation}
\varrho(x,t)=\lambda \Id_{(-\infty,(1-\lambda-\rho)t]}(x)+\rho \Id_{((1-\lambda-\rho)t,+\infty)}(x).
\end{equation}
Thus the discontinuity at the origin, also known as shock discontinuity or shortly shock, is preserved and it moves with speed $v_s=1-\lambda-\rho$.

In this paper we study the evolution of a second class particle which started at the origin and determine its limiting process. The second class particle can be seen as the discrepancy between two TASEP configurations $\eta$ and $\tilde\eta$, where at time $0$, $\eta_0(x)=\tilde \eta_0(x)$ for all $x\neq 0$ and $\eta_0(0)=0$ while $\tilde\eta_0(0)=1$. Under basic coupling  the configurations $\eta_t$ and $\tilde\eta_t$ differ exactly in one site (see e.g.~\cite{Li99}), which is the position of the second class particle, denoted by $X^{\rm 2nd}(t)$. Second class particle are very useful in presence of shock discontinuities since they are attracted by them so that the position of the shock can be identified by the position of the second class particle, see~\cite{Li99}, Chapter 3.

The case of Bernoulli-Bernoulli initial condition, namely where each site is initially occupied independently with probability $\lambda$ for sites in $\Z_-$ and $\rho$ for sites in $\Z_+$, has been extensively studied long time ago. First it was shown that the fluctuations of $X^{\rm 2nd}(t)$ with respect to macroscopic position $(1-\lambda-\rho)t$ are in the $t^{1/2}$ scale~\cite{Fer90}, while in~\cite{FF94b} it is proven that the fluctuations are Gaussian and the limit process is a Brownian motion, see also~\cite{PG90,DMKPS89} for related results. The reason for  the Brownian behavior lies in the fact that $X^{\rm 2nd}(t)$ can be directly related to the random initial data,
see Theorem 1.1 in~\cite{FF94b}.

For shocks with non-random initial conditions the situation is different, since in this case the typical KPZ fluctuations coming from the dynamics are relevant (unlike the previous case). In~\cite{FGN17} we have proven that for initial conditions with non-random densities $\lambda<\rho$, the fluctuations of the second class particle starting from the origin is in the $t^{1/3}$ scale and its distribution is given by the difference of two independent random variable with (rescaled) GOE Tracy-Widom distributions. A similar structure was previously shown for the fluctuation of a related quantity, the competition interface in the last passage percolation (LPP) model, see~\cite{FN16,N17} for non-random initial conditions and~\cite{FMP09} for the Bernoulli case. The structure of independence at shocks was studied in~\cite{FN13} and for multishocks in~\cite{FN19}. This arises also as limit of the soft shock process introduced in~\cite{FN14}, see~\cite{QR18}. More recently using a symmetry theorem in multi-colored TASEP~\cite{BB19}, identities for the one-point distribution of a second class particle were derived for some class of initial conditions (essentially finite perturbations of the step initial conditions) and their asymptotic fluctuations analyzed in various scaling~\cite{BF22}. For a certain shock, KPZ fluctuations of the second class particle in ASEP have been obtained in~\cite{N20CMP}.

The results of this paper are as follows. First of all we extend the result of~\cite{FGN17} to the convergence of joint distributions. Secondly we do it without using the mapping to LPP and the competition interface therein. In particular we do not have to deal with a random time change of the competition interface as in~\cite{FGN17}.  We thus  provide a direct  understanding of the second class particle at the shock. To do so,  we obtain  in Theorem~\ref{thmDistrSecondClass} an explicit expression for the distribution function of the second class particle in terms of the distribution of two height functions associated with TASEP configurations. Using this theorem as starting point, we are able to derive the asymptotic second class particle process in a relatively short way. In Theorem~\ref{thmMain} we derive a general result based on some assumptions. As illustration we present the case of a shock  created by \textit{deterministic} initial data (see \eqref{eq5.5}) between density $\lambda$ and density $\rho, \lambda<\rho$: Setting
\begin{equation}
{\cal H}^-(u)=\chi_-^{2/3} 2^{4/3} {\cal A}_1(\chi_-^{-1/3}2^{-5/3}u),\quad {\cal H}^+(u)=\chi_+^{2/3} 2^{4/3} \widetilde {\cal A}_1(\chi_+^{-1/3}2^{-5/3}u),
\end{equation}
with ${\cal A}_1$ and $\widetilde {\cal A}_1$ two  independent Airy$_1$ processes~\cite{Sas05,BFPS06}, $\chi_-=\lambda(1-\lambda)$, $\chi_+=\rho(1-\rho)$, we  show in Corollary~\ref{truecor}
\begin{equation}
\lim_{t\to\infty}\frac{X^{\rm 2nd}(t+\tau t^{2/3})-v_s (t+\tau t^{2/3})}{t^{1/3}}=\frac{{\cal H}^-((\lambda-\rho)\tau)-{\cal H}^+((\rho-\lambda)\tau)}{2(\rho-\lambda)}
\end{equation}
in the sense of finite-dimensional distribution (recall $v_s= 1-\lambda-\rho$ is the speed of the shock). Thus the second class particle process is asymptotically distributed as the difference of two independent Airy$_1$ processes.

The rest of the paper is organized as follows. In Section~\ref{SectFiniteTimeResult} we derive the finite-time result, Theorem~\ref{thmDistrSecondClass}. In Section~\ref{sectAsymptotics} we prove the asymptotic result, Theorem~\ref{thmMain} and apply it to the special case exactly constant densities in Corollary~\ref{corAspecial}. In Appendix~\ref{AppFlat} we give the outline on how to derive the limit to the Airy$_1$ process for densities different from $1/2$, using the method of the KPZ fixed point introduced in~\cite{MQR17}. This is a well expected result, but for general densities it is not available in the literature, except for the one-point distribution case~\cite{FO17}. Appendix~\ref{AppStepIC} contains some known results on step initial conditions, which are used as input.

\section{Distribution of the second class particle}\label{SectFiniteTimeResult}
In this section we introduce the model and derive the first main theorem, namely a new identity on the distribution of the second class particle in terms of two height functions.

\subsection{The model}
We consider the totally asymmetric simple exclusion process (TASEP) on $\Z$ with initial condition generating a shock at the origin and one second class particle starting at the origin. We determine the limiting process of the second class particle.

Particles in TASEP have jump rate one and they jump to their neighboring right site, conditioned on being free. One can graphically construct TASEP via i.i.d. Poisson processes $\mathcal{P}(x),x\in \Z,$ of rate one:   if there is a particle at $x$ and no particle at $x+1$ when $ \mathcal{P}(x)$ has a jump,
the particle at $x$ jumps to $x+1$. The \textit{basic coupling }couples TASEPs with different initial configurations together  by making them use the same Poisson processes $\mathcal{P}(x),x\in \Z.$ For a TASEP $\eta$, we denote by $\eta_t(x)\in\{0,1\}$ the occupation variable at time $t\in\R_+$ and site $x\in\Z$ and by $\eta_0$ the initial configuration. We associate a height function by setting
\begin{equation}
h(0,0)=0,\quad h(x+1,t)-h(x,t)=1-2\eta_t(x).
\end{equation}

If we have a local minimum at $x$, that is, if $h(x,t)-h(x-1,t)=-1$ and $h(x+1,t)-h(x,t)=1$, then we have $\eta_t(x-1)=1$ and $\eta_t(x)=0$. When particle at site $x-1$ jumps to site $x$, the local minimum at $x$ becomes a local maximum. Thus a Poisson event (a jump trial) at site $x-1$ will attempt to increase the height function at site $x$ by $2$.

Notation: if we have a height function $h^*$ for some symbol $*$, then the corresponding particle configuration is denoted by $\eta^*$.

A second class particle can be seen as a discrepancy between two particle configurations: consider $\eta$ and $\tilde\eta$ two configurations where $\eta(x)=\tilde\eta(x)$ for all $x\in\Z\setminus\{x_0\}$ and $\eta(x_0)=0$, $\tilde\eta(x_0)=1$. Then we say that we have a second class particle at site $x_0$. We couple TASEP with different initial configuration by the basic coupling, that it, we use the same Poisson events for the jump trials. In particular, if we start with a single discrepancy, then for all times there will be a single discrepancy, which we denote by $X^{\rm 2nd}(t)$ and call it the position of the second class particle associated with the configurations $\eta_t$ and $\tilde\eta_t$. We have
\begin{equation}
X^{\rm 2nd}(t)=\sum_{x\in\Z} x \Id_{\eta_t(x)\neq \tilde\eta_t(x)}.
\end{equation}

Now consider two initial configurations $\eta_0$ and $\tilde\eta_0$ with a discrepancy at $0$ (i.e., $\eta_0(0)=0$ and $\tilde\eta_0(0)=1$), i.e., we have a second class particle starting at the origin: $X^{\rm 2nd}(0)=0$. Then we have the relation
\begin{equation}\label{eq1.2}
\begin{aligned}
\tilde h(x,t)&=h(x,t), &\textrm{for all }x\leq X^{\rm 2nd}(t),\\
\tilde h(x,t)&=h(x,t)-2, &\textrm{for all }x> X^{\rm 2nd}(t).
\end{aligned}
\end{equation}
Define the configurations
\begin{equation}\label{eq2.5}
\begin{aligned}
\eta^-_0(x)&=\eta_0(x)\Id_{x<0}+0\Id_{x\geq 0},\\
\eta^+_0(x)&=\eta_0(x)\Id_{x\geq 0}+1\Id_{x<0},\\
\tilde \eta^+_0(x)&=\tilde\eta_0(x)\Id_{x\geq 0}+1\Id_{x<0}.
\end{aligned}
\end{equation}
In particular, note that $\eta^-_0(x)\leq \tilde\eta^+_0(x)$ for all $x\in\Z$ and by basic coupling this holds for any time $t$, namely
\begin{equation}\label{eq1.4}
\eta^-_t(x)\leq \tilde\eta^+_t(x)\textrm{ for all }x\in\Z.
\end{equation}
In terms of height function we have, see Figure~\ref{FigIC},
\begin{equation}
\begin{aligned}
h^-(x,0)&=h(x,0)\Id_{x<0}+|x|\Id_{x\geq 0},\\
h^+(x,0)&=h(x,0)\Id_{x\geq 0}+|x|\Id_{x<0},\\
\tilde h^+(x,0)&=\tilde h(x,0)\Id_{x\geq 0}+|x|\Id_{x<0}=(h(x,0)-2)\Id_{x\geq 1}+|x|\Id_{x\leq 0}.
\end{aligned}
\end{equation}
\begin{figure}
\begin{center}
\includegraphics[height=7cm]{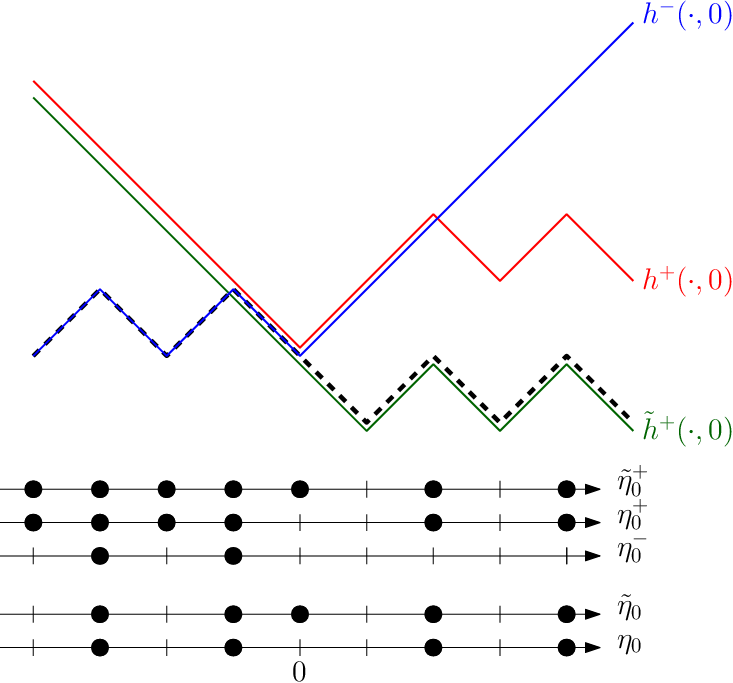}
\caption{The initial height function $h(x,0)$ is the dashed thick line, which matches with $h^-(x,0)$ for $x<0$ and with $h^+(x,0)$ for $x>0$. The second class particle starts at $x=0$. The small vertical shifts are just for making the illustration clearer and avoid overlapping height functions with different initial conditions.}
\label{FigIC}
\end{center}
\end{figure}
It is well known that
\begin{equation}\label{eq1.5}
h(x,t)=\min\{h^-(x,t),h^+(x,t)\},\quad \tilde h(x,t)=\min\{h^-(x,t),\tilde h^+(x,t)\}.
\end{equation}

\subsection{Distribution of the second class particle}
By definition of the second class particle, we have
\begin{equation}
\{X^{\rm 2nd}(t)=x\}=\{h(x,t)=\tilde h(x,t), h(x+1,t)>\tilde h(x+1,t)\}.
\end{equation}
Our first result is an expression of the distribution of the second class particle in terms of the height function $h^-$ and $\tilde h^+$ only.
\begin{thm}\label{thmDistrSecondClass}
With the above notations, for $0\leq t_1<t_2<\ldots<t_m$ and $x_1,\ldots,x_n\in\Z$, we have the identity
\begin{equation}
\Pb(X^{\rm 2nd}(t_i)\geq x_i,1\leq i\leq m)=\Pb(h^-(x_i,t_i)\leq \tilde h^+(x_i,t_i),1\leq i\leq m).
\end{equation}
\end{thm}
\begin{proof}
We prove that for any $x\in\Z$ and any $t\geq 0$, $\{X^{\rm 2nd}(t)\geq x\} = \{h^-(x,t)\leq \tilde h^+(x,t)\}$ under the basic coupling. For this we prove the inclusions in both directions.

(a) Assume that $h^-(x,t)\leq \tilde h^+(x,t)$, then also $h^-(x,t)\leq h^+(x,t)$ and by \eqref{eq1.5} $h(x,t)=h^-(x,t)=\tilde h(x,t)$, which implies by \eqref{eq1.2} that $X^{\rm 2nd}(t)\geq x$.

(b) Let $Y(t)$ be as in Proposition~\ref{y=x}.  Using Propositions~\ref{y=x},~\ref{yprop}, and~\ref{ineqprop} below, we conclude
\begin{equation}
\begin{aligned}
&\{X^{\rm 2nd}(t)\geq x\}=\bigcup_{z\geq x}\{Y(t)=z\}\subseteq \bigcup_{z\geq x}  \{h^{-}(z,t)=\tilde{h}^{+}(z,t)\}\\
&\subseteq \bigcup_{z\geq x} \bigcap_{y\leq z}\{h^-(y,t)\leq \tilde h^+(y,t)\}\subseteq \{h^-(x,t)\leq \tilde h^+(x,t)\}.
\end{aligned}
\end{equation}
\end{proof}

\begin{prop}\label{ineqprop}
Let $x\in \Z$. We have for all $t\geq 0$
\begin{equation}
\{h^{-}(x,t)= \tilde{h}^{+}(x,t)\}\subseteq\bigcap_{y\leq x}\{h^{-}(y,t)\leq \tilde{h}^{+}(y,t)\}.
\end{equation}
\end{prop}
\begin{proof}
We prove by contraposition. Assume there is a $y\leq x$ such that $h^{-}(y,t)>\tilde{h}^{+}(y,t)$. Take the largest such $y$. If $y=x$, clearly $h^{-}(x,t)\neq \tilde{h}^{+}(x,t).$ If not, we must have $h^{-}(y+1,t)= \tilde{h}^{+}(y+1,t)$. This however can only happen if $\eta^{-}_{t}(y)=1,\tilde{\eta}^{+}_{t}(y)=0,$ which is impossible since $\eta^{-}_{t}\leq \tilde{\eta}^{+}_{t}$ coordinate wise, see \eqref{eq1.4}.
\end{proof}
Since $\eta^{-}_{0}\leq \tilde{\eta}^{+}_{0}$ coordinate wise, we can define a TASEP with first class particles at positions $\{x:\eta^{-}_{0}(x)=1\}$ and second class particles at positions $\{x:\tilde{\eta}^{+}_{0}(x)-\eta^{-}_{0}(x)=1\}$. Note that by \eqref{eq2.5} there is initially a second class particle at position $0$. We call $Y(t)$ the position of this second class particle at time $t$. As the next proposition shows,  $Y$  is just $X^{\rm 2nd}$ in disguise, but $Y$ will be useful to prove Theorem~\ref{thmDistrSecondClass}.
\begin{prop}\label{y=x} Under the basic coupling, we have
\begin{equation}
(Y(t),t\geq 0)=(X^{\rm 2nd}(t),t\geq 0).
\end{equation}
\end{prop}
\begin{proof}
Note that the evolution of $X^{\rm 2nd}$ (in fact, of any second class particle in TASEP) does not change if some first class particles initially to the right of $X^{\rm 2nd}$ are turned into second class particles, as $X^{\rm 2nd}$ gets blocked by both types of particles. Likewise, the evolution of $X^{\rm 2nd}$ does not change if new second class particles are added to its left.

In our case, the initial particle configuration from which $Y$ starts  is obtained from the initial configuration from which $X^{\rm 2nd}$  starts by turning all particles $\{x>0:\eta(x)=1\}$, i.e., all particles to the right of $X^{\rm 2nd}(0),$ into second class particles, and by adding second class particles at positions $\{x<0: \eta(x)=0\}$. By the preceding argument, this does not affect the evolution of  $X^{\rm 2nd}$ and thus the evolution of $X^{\rm 2nd}$ and $Y$ are identical.
\end{proof}
\begin{prop}\label{yprop}
We have
\begin{align}
\label{first}&h^{-}(Y(t),t)=\tilde{h}^{+}(Y(t),t),\\
&\label{second}h^{-}(Y(t)+1,t)=\tilde{h}^{+}(Y(t)+1,t)+2.
\end{align}
\end{prop}
\begin{proof}
The statement is true by definition for $t=0$ since $Y(0)=0$ and $\tilde{\eta}^{+}(0)=1>\eta^{-}(0)=0.$

Let $\tau$ be a jump time of $Y$, and assume \eqref{first}, \eqref{second} hold at $\tau^{-}$ (infinitesimal time before $\tau$), see Figure~\ref{FigJumps}.
\begin{figure}
\begin{center}
\includegraphics[height=4.5cm]{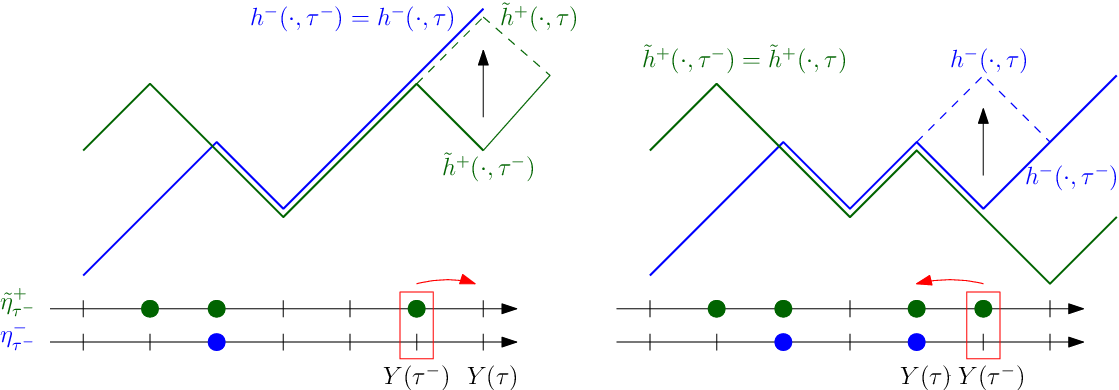}
\caption{Illustration of the configurations and height function. A clock rings at position $Y(\tau^-)$ at time $\tau$. (Left): the second class particle jumps   to the right; (Right): the second class particle jumps to the left.}
\label{FigJumps}
\end{center}
\end{figure}
We first consider the case  that it jumps to the right: $Y(\tau)=Y(\tau^-)+1$. Then $h^{-}(\cdot,\tau)=h^{-}(\cdot,\tau^{-})$ and
\begin{equation}
\begin{aligned}
\tilde{h}^{+}(Y(\tau),\tau)&= \tilde{h}^{+}(Y(\tau),\tau^{-})+2= \tilde{h}^{+}(Y(\tau^{-})+1,\tau^{-})+2 \\&=h^{-}(Y(\tau^{-})+1,\tau^{-})=h^{-}(Y(\tau),\tau),
\end{aligned}
\end{equation}
 so that \eqref{first} still holds at time $\tau$. Since $\eta^{-}_{\tau}(Y(\tau))=0$ and $\tilde{\eta}^{+}_{\tau}(Y(\tau))=1$ we get
\begin{equation}
\begin{aligned}
 h^{-}(Y(\tau)+1,\tau)&=h^{-}(Y(\tau),\tau)+1\\ &=\tilde{h}^{+}(Y(\tau),\tau)+1=\tilde{h}^{+}(Y(\tau)+1,\tau)+2,
\end{aligned}
\end{equation}
showing \eqref{second} at time $\tau$.

Let now   $Y$ jump to the left at time $\tau$. This implies
\begin{equation}
\begin{aligned}
&\tilde{h}^{+}(\cdot,\tau)=\tilde{h}^{+}(\cdot,\tau^{-}),\\
&h^{-}(Y(\tau^{-}),\tau)=h^{-}(Y(\tau^{-}),\tau^{-})+2.
\end{aligned}
\end{equation}
 By assumption, we have
\begin{equation}
\tilde{h}^{+}(Y(\tau^{-}),\tau^{-})=h^{-}(Y(\tau^{-}),\tau^{-}).
\end{equation}
This shows $h^{-}(Y(\tau^{-}),\tau)=\tilde{h}^{+}(Y(\tau^{-}),\tau)+2.$
 Since $Y(\tau)+1=Y(\tau^{-})$, \eqref{second} holds at time $\tau$.
 To show \eqref{first}, note that since \eqref{first} holds at $\tau^{-}$,  by Proposition~\ref{ineqprop} we have \begin{equation}h^{-}(Y(\tau),\tau)=h^{-}(Y(\tau),\tau^{-})\leq \tilde{h}^{+}(Y(\tau),\tau^{-})=\tilde{h}^{+}(Y(\tau),\tau)\end{equation}
 Now for arbitrary $x,t$, the conditions $h^{-}(x,t)\leq\tilde{h}^{+}(x,t)$ and at the same time $h^{-}(x+1,t)=\tilde{h}^{+}(x+1,t)+2$ imply $h^{-}(x,t)=\tilde{h}^{+}(x,t)$. Applying this with $x=Y(\tau)$ and $t=\tau$ shows \eqref{first} at time $\tau$.
\end{proof}

\section{Asymptotic results}\label{sectAsymptotics}
We briefly review some material from~\cite{BF22}, Section~4. It is shown there (and was known before) that, under the aforementioned basic coupling of using the same Poisson processes,  for any time $\tau\in [0,t]$ we have the identity
\begin{equation}\label{minid}
h(x,t)=\min_{y\in\Z}\{h(y,\tau)+h^{\rm step}_{y,\tau}(x,t)\},
\end{equation}
where $h^{\rm step}_{y,\tau}(x,t)$ is the height function starting at time $\tau$ from $h^{\rm step}_{y,\tau}(x,\tau)=|x-y|$ (this is the height function of the step initial data $\mathbf{1}_{\mathbb{Z}_{\leq y}}$ modulo a shift).
The identity \eqref{minid} allows us to define (not necessarily unique) backward geodesics, which go back in time from time point  $t$ to time $0$.
\begin{defin}\label{DefGeodesics}
Any backwards trajectory $\{\x(\tau),\tau:t\to 0\}$ with $\x(t)=x$ and satisfying
\begin{equation}\label{eq2.2}
h(x,t)=h(\x(\tau),\tau)+h^{\rm step}_{\x(\tau),\tau}(x,t)
\end{equation}
for all times $\tau\in [0,t]$ is called a \emph{backwards geodesics}.
\end{defin}

\subsection{Large time results}\label{SectAssumptions}
Let us consider a generic initial condition $h(x,0)$ (it could be even random) satisfying the following three assumptions:
\begin{itemize}
\item[(a)] \emph{Initial macroscopic shock around the origin}: assume that for some $0<\lambda<\rho<1$,
\begin{equation}
\lim_{x\to-\infty} x^{-1}h(x,0)=1-2\lambda,\quad \lim_{x\to\infty} x^{-1}h(x,0)=1-2\rho.
\end{equation}
\item[(b)] \emph{Control of the end-point of the backwards geodesics} starting at $v_s t$ with $v_s=1-\lambda-\rho$: let $\x^+$ (resp.\ $\x^-$) any backwards geodesics starting at $(v_s t,t)$ with initial profile $\tilde h^+(\cdot,0)$ (resp.\ $h^-(\cdot,0)$). Then assume for some $\delta>0$, we have
\begin{equation}
\lim_{t\to\infty}\Pb(\x^-(0)\leq -\delta t,\x^+(0)\geq \delta t) = 1.
\end{equation}
\item[(c)] \emph{Limit process under KPZ scaling} of $\tilde h^+,h^-$ with \emph{constant law along characteristics}. Let $\tilde\nu<1$ and  $|\theta|\leq t^{\tilde\nu}$. There exist limiting processes ${\cal H}^-(u)$ and ${\cal H}^+(u)$ whose distribution is continuous in $u$, such that
\begin{equation}\label{eq2.3}
\begin{aligned}
&\lim_{t\to\infty} \frac{h^-(v_s t+u t^{2/3}+(1-2\lambda)\theta,t+\theta)-{\bf h}^-(u,\theta,t)}{-t^{1/3}}={\cal H}^-(u),\\
&\lim_{t\to\infty} \frac{\tilde h^+(v_s t+u t^{2/3}+(1-2\rho)t,t+\theta)-{\bf h}^+(u,\theta,t)}{-t^{1/3}}={\cal H}^+(u),
\end{aligned}
\end{equation}
where
\begin{equation}\label{eq2.4}
\begin{aligned}
{\bf h}^-(u,\theta,t)&=(1-\lambda-\rho+2\lambda \rho) t+ (1-2\lambda) u t^{2/3}+(1-2\lambda(1-\lambda))\theta,\\
{\bf h}^+(u,\theta,t)&=(1-\lambda-\rho+2\lambda \rho) t+ (1-2\rho) u t^{2/3}+(1-2\rho(1-\rho))\theta
\end{aligned}
\end{equation}
are what we expect macroscopically from the solution of the Burgers equation. Below we will use the notations
\begin{equation}
\mu_s=1-\lambda-\rho+2\lambda \rho,\quad \chi_-=\lambda(1-\lambda),\quad \chi_+=\rho(1-\rho).
\end{equation}
\end{itemize}

\begin{thm}\label{thmMain}Let $\tau_k,s_k \in \R,k=1,\ldots,m$.
Under the above assumptions, with $X^{\rm 2nd}(t)$ denoting the position of the second class particle at time $t$ starting from the origin, we have
\begin{equation}\label{eq3.8b}
\begin{aligned}
&\lim_{t\to\infty} \Pb\left(\bigcap_{k=1}^m\Big\{X^{\rm 2nd}(t+\tau_k t^{2/3})\geq v_s (t+\tau_k t^{2/3}) + s_k t^{1/3}\Big\}\right) \\
&=\Pb\left(\bigcap_{k=1}^m \{{\cal H}^-((\lambda-\rho)\tau_k)-{\cal H}^+((\rho-\lambda)\tau_k)\geq 2(\rho-\lambda)s_k\}\right),
\end{aligned}
\end{equation}
where the processes ${\cal H}^+$ and ${\cal H}^-$ are \emph{independent}.
\end{thm}

Let us illustrate Theorem~\ref{thmMain} in a concrete example. It is sometimes useful to describe a TASEP configuration in terms of the position of its labelled particles (see Appendix~\ref{AppFlat}). We denote by $X_t(n)$ the position of particle $n$ at time $t$ and use the right-to-left convention, namely $X_t(n+1)< X_t(n)$ for all $n$ and $t$. The initial condition with non-random densities $\lambda$ (resp.\ $\rho$) to the left (resp.\ right) of the origin can be then described by the initial condition
\begin{equation}\label{eq5.5}
X_0(n)=\left\{\begin{array}{ll}
-\lfloor n/\lambda\rfloor, &n\geq 1,\\
-\lfloor n/\rho\rfloor,& n\leq 0.
\end{array}\right.
\end{equation}

\begin{cor}\label{truecor}
Let $X^{\rm 2nd}(t)$ denote the position of the second class particle at time $t$ starting from the origin with initial condition \eqref{eq5.5}. Then we have
\begin{equation}
\begin{aligned}
&\lim_{t\to\infty} \Pb\left(\bigcap_{k=1}^m\Big\{X^{\rm 2nd}(t+\tau_k t^{2/3})\geq v_s (t+\tau_k t^{2/3}) + s_k t^{1/3}\Big\}\right) \\
&=\Pb\left(\bigcap_{k=1}^m \{{\cal H}^-((\lambda-\rho)\tau_k)-{\cal H}^+((\rho-\lambda)\tau_k)\geq 2(\rho-\lambda)s_k\}\right),
\end{aligned}
\end{equation}
where
\begin{equation}
{\cal H}^+(u)=\chi_+^{2/3} 2^{4/3} \widetilde {\cal A}_1(\chi_+^{-1/3}2^{-5/3}u),\quad {\cal H}^-(u)=\chi_-^{2/3} 2^{4/3} {\cal A}_1(\chi_-^{-1/3}2^{-5/3}u)
\end{equation}
are \emph{independent} rescaled Airy$_1$ processes (see~\cite{BFPS06,Sas05} for the definition of the Airy$_1$ process and Appendix~\ref{AppFlat} for how the it appears as a scaling limit of TASEP for generic particle density in $(0,1)$).
\end{cor}

Theorem~\ref{thmMain} is proven starting with Theorem~\ref{thmDistrSecondClass} and taking the large time limit. To get independence of the processes ${\cal H}^-$ and ${\cal H}^+$ the two key ingredients are slow decorrelation and localization of backwards paths that we discuss below. After that we will complete the proof of Theorem~\ref{thmMain} and Corollary~\ref{corAspecial}.

\begin{rem}
Theorem~\ref{thmMain} is proven under the scaling assumption \eqref{eq2.3}, with $1/3$ exponent for the fluctuations. One could also consider Bernoulli-$\lambda$ initial conditions to the left and Bernoulli-$\rho$ to the right of the origin. In that case, fluctuations are dominated by the ones of the initial condition around the start of the characteristic lines, so the fluctuation exponents will be $1/2$ and the limit processes Brownian motions. Thus one would expect that the limit process is a difference of two independent Brownian motions. As this was proven already in 1994  in~\cite{FF94b} with other methods, we do not purse this here.
\end{rem}

\subsection{Backwards paths}
In~\cite{BF22} an explicit construction of backwards paths having the property \eqref{eq2.2} was given, which we simplify here\footnote{There is a change w.r.t.   Definition~4.1 of~\cite{BF22}, namely in the case of a maximum the backwards path needs to move to one of its neighboring points, which  was assumed but not stated in the proof of Proposition~4.2 of~\cite{BF22}.}.
\begin{defin}\label{DefinBackwardsPaths}
We construct a \emph{backwards path} \mbox{$\{\x(s),s:t\to 0\}$} starting from  $\x(t)=x$ as follows. Let $s$ be the time of the last Poisson event  during $[0,t]$ at position $x-1$ (we set $s=0$ if no such event occurred). Then we define $\x(\ell)=x,s\leq \ell\leq t,$ and $\x(s_{-} )$  is given as follows\footnote{In~\cite{BF22} we said that a Poisson event at $y$ changes the height function at $y$ from a local minimum to a local maximum, while to be consistent with Section~\ref{SectFiniteTimeResult} we have here that the change of height function in $y$ is due to a Poisson event at $y-1$, which is associated to the particle which was at $y-1$.}:(a) If $h(\x(s),s)=h(\x(s),s_-)+2$ (a minimum became a maximum), then $\x(s_-)=\x(s)$. (b) If $h(\x(s),s)=h(\x(s),s_-)$, then $\x(s_-)$ is any nearest neighbor of $\x(s)$ where the height function has value $h(\x(s),s)-1$ (in case that it is a maximum, we have two possible choices), see Figure~\ref{FigBackwards}.
\end{defin}

\begin{figure}
\begin{center}
\includegraphics[height=4cm]{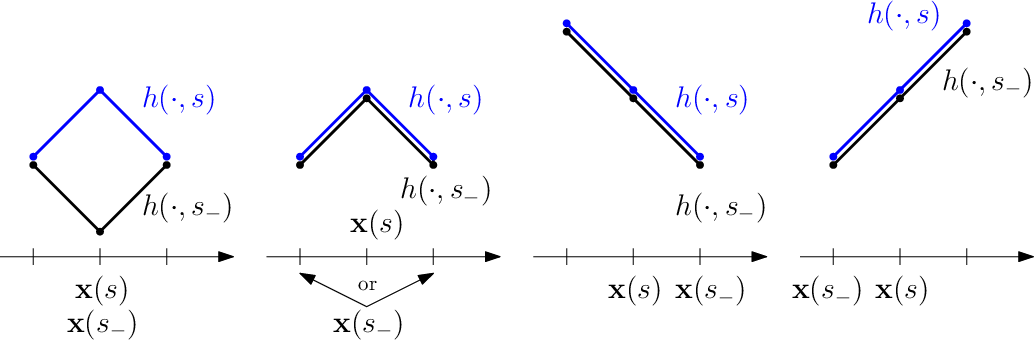}
\caption{Construction of the backwards path as in Definition~\ref{DefinBackwardsPaths}. A Poisson event happens at $(x(s),s)$. The height function at time $s$ is in blue, while the height function just before is in black.}
\label{FigBackwards}
\end{center}
\end{figure}

In Proposition~4.2 of~\cite{BF22} it was shown  that backwards paths are in particular also backwards geodesics. Given a backwards path $\{\x(\tau),\tau:t\to 0\}$ with $\x(t)=x$ for the height function $h$,  for any $\tilde t<t$, the path $\{\tilde \x(\tau),\tau:\tilde t\to 0\}$ with $\tilde\x(\tilde t)=\x(\tilde t)$ is also a backwards path for the height function $h$.

For the localization of the backwards paths we will use two properties about the ordering.
\begin{lem}\label{lemGeodesics}
For a given height function $h$, consider the two rightmost backwards paths $\{\x_1(\tau),\tau:t\to 0\}$ with $\x_1(t)=x_1$ and $\{\x_2(\tau),\tau:t\to 0\}$ with $\x_2(t)=x_2$. If $x_2>x_1$, then $\x_2(\tau)\geq \x_1(\tau)$ for all $\tau \in [0,t]$ under the basic coupling.
\end{lem}
\begin{proof}
Since jumps are nearest neighbor and a.s.\ no two jumps occur at the same time,  $ \x_2$ cannot jump strictly to the left of $\x_1$ without being identical to $\x_1$ at some time $s$. However, if there is a time $s$ such that $\x_2(s)= \x_1(s),$ then
$\x_2(\tau)= \x_1(\tau),\tau\leq s.$ Therefore  we can never have $\x_2(\tau)< \x_1(\tau).$
\end{proof}

A second property that we will use is the ordering of backwards paths with different initial condition starting at the same position.
\begin{lem}\label{lemOrderingGeoStepIC}
Consider an otherwise arbitrary height function $h$  satisfying  $h(x,0)=x$ for $x\geq 0$. Let $h^{\rm step}(x,0)=|x|$ for all $x$. Consider the rightmost backwards paths starting from $x$ at time $t$, where we denote the one for $h$ by $\{\x(\tau),\tau:t\to 0\}$ and the one for $h^{\rm step}$ by $\{\x^{\rm step}(\tau),\tau:t\to 0\}$. Thus $\x(t)=\x^{\rm step}(t)=x$.  Then, for any $\tau\in [0,t]$, $\x(\tau)\leq \x^{\rm step}(\tau)$.
\end{lem}
\begin{proof} At time $t$ we have equality. What we have to rule out is the possibility that there is a $\tau<t$ such that $\x(\tau)=\x^{\rm step}(\tau)$ but either (a) $\x(\tau_-)=\x(\tau)+1$ and $\x^{\rm step}(\tau_-)\leq \x^{\rm step}(\tau)$, or (b) $\x(\tau_-)\geq \x(\tau)$ and $\x^{\rm step}(\tau_-)=\x^{\rm step}(\tau)-1$. We write $x=\x(\tau)=\x^{\rm step}(\tau)$.

In order (a) to happen, we need to have $\eta^{\rm step}_\tau(x)=0$ and $\eta_\tau(x)=1$, while (b) happens if $(\eta^{\rm step}_\tau(x-1),\eta^{\rm step}_\tau(x))=(0,0)$ and $(\eta_\tau(x-1),\eta_\tau(x))\neq (0,0)$. But at time $t=0$, $\eta^{\rm step}_0\geq \eta_0$ coordinate wise and this remains true at any later time by the basic coupling. Thus (a) and (b) can not occur.
\end{proof}

Let us recall one result on the localization of the geodesics for step initial conditions.

\begin{prop}[Proposition 4.9 of~\cite{BF22}]\label{PropLocStepIC}
Consider TASEP with step initial conditions, i.e., $h(x,0)=|x|$. Let $\x(t)=\alpha t$ be the starting position of a backwards geodesics with $\alpha\in (-1,1)$. Then, uniformly for all $t$ large enough,
\begin{equation}
\Pb(|\x(\tau)-\alpha\tau|\leq u t^{2/3}\textrm{ for all }0\leq \tau\leq t)\geq 1-C e^{-c u^2}
\end{equation}
for $u>0$ and  some constants $C,c>0$.
\end{prop}

\begin{figure}
\begin{center}
\includegraphics[height=5cm]{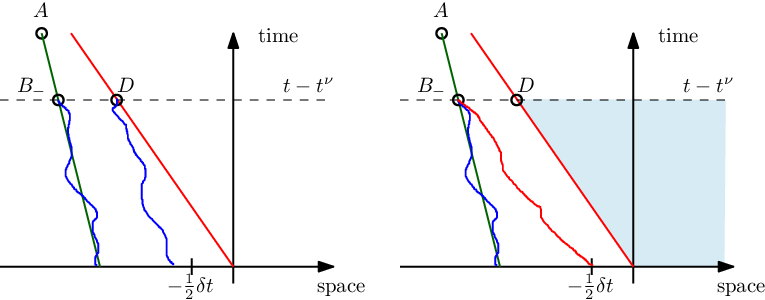}
\caption{Localization strategy: (left) with the same initial conditions, geodesics to $B_-$ are to the left of the right-most geodesics to the point $D=(v_s (t-t^\nu),t-t^\nu)$, which by hypothesis is to the left of $-\tfrac12\delta t$ with high probability. (right) if the geodesics to $B_-$ is to the left of $-\tfrac12\delta t$, then it is to the left of the right-most one with step initial condition to $B_-$, which by hypothesis is not fluctuating more than $\Or(t^{2/3})$, thus it stays to the left of the line $(0,0)$ to $D$.}
\label{FigLocalization}
\end{center}
\end{figure}

In the following statements we will use the space-time points
\begin{equation}\label{eqAB}
\begin{aligned}
A&=(v_s(t+\tau t^{2/3})+s t^{1/3},t+\tau t^{2/3}),\\
B_-&=(v_s (t+\tau t^{2/3})-(1-2\lambda)(\tau t^{2/3}+t^\nu),t-t^\nu),\\
B_+&=(v_s (t+\tau t^{2/3})-(1-2\rho)(\tau t^{2/3}+t^\nu),t-t^\nu),
\end{aligned}
\end{equation}
where $\nu\in (2/3,1)$.
Under the assumptions of Section~\ref{SectAssumptions}, we control the localization of the backwards paths starting at $B_-$ (and similarly to $B_+$).
\begin{prop}\label{propLocalization}
Let $\x^-$ be any backwards paths for $h^-$ starting at $B_-$ and $\x^+$ any backwards paths for $\tilde h^+$ starting at $B_+$. Then, for any $\nu\in (2/3,1)$,
\begin{equation}\label{eq3.7}
\lim_{t\to\infty} \Pb(\x^-(\ell)< v_s\ell\textrm{ for all }0\leq \ell\leq t-t^\nu)=1
\end{equation}
and
\begin{equation}\label{eq3.8}
\lim_{t\to\infty} \Pb(\x^+(\ell)> v_s\ell\textrm{ for all }0\leq \ell\leq t-t^\nu)=1.
\end{equation}
\end{prop}
\begin{proof}
We prove \eqref{eq3.7} about backward paths  for $h^-$ starting at $B_-$, the proof of \eqref{eq3.8} is analogous.   We call $\x^{-}_1$   the rightmost backwards path of $h^-$ starting at $D=(v_s (t-t^\nu),t-t^\nu)$
and  $\x^{-}_2$   the rightmost backwards path of $h^-$ starting at $B_-$, see Figure~\ref{FigLocalization} (left). Denote by $G$ the event that  $\x^{-}_1$ stays to the left of $-\tfrac12\delta t$. By Assumption (b), since $t-t^\nu\geq t/2$ for all $t$ large enough, we have that any geodesics, thus also  $\x^{-}_1$ stays to the left of $-\tfrac12\delta t$  with high probability. Note that the $x-$coordinate of $B_-$ lies to the left of the $x-$coordinate of $D$. Thus, by the order of backwards paths (see Lemma~\ref{lemGeodesics}), $\x^{-}_2$ stays to the left $\x^{-}_1$, as in Figure~\ref{FigLocalization} (left).

Let now $\tilde{\x}^{-}_2$  be the rightmost backwards path starting at $B_-$ with initial configuration given by the height function of the particle configuration $\eta^{-}\mathbf{1}_{x \leq -\delta t/2}.$ This means that    $\tilde{\x}^{-}_2$   differs from
$\x^{-}_2 $ by having no particles to the right of $-\delta t/2$ initially. However, on the event $G$, we have precisely $\x^{-}_2 =\tilde{\x}^{-}_2$.

Noting that up to a spatial shift by $-\delta t/2$ we are exactly in the setting of Lemma~\ref{lemOrderingGeoStepIC}. Applying  Lemma~\ref{lemOrderingGeoStepIC} to $ \tilde{\x}^{-}_2$ we get that the rightmost backwards path starting at $B_-$ with step initial condition $h^{\rm step}_{-\delta t/2,0}(\cdot,0)$ stays to the right of $ \tilde{\x}^{-}_2$ and thus, on the event $G$, to the right of $ \x^{-}_2$.   By Proposition~\ref{PropLocStepIC}, the rightmost backwards path starting at $B_-$ with step initial condition $h^{\rm step}_{-\delta t/2,0}(\cdot,0)$ has $\mathcal{O}(t^{2/3})$ fluctuations around the straight line joining $(-\delta t/2,0)$ and $B_-$. In particular, since $\nu>2/3$, it stays to the left of the line $\{(v_s \ell,\ell),0\leq \ell\leq t-t^\nu\}$ with high probability, see Figure~\ref{FigLocalization} (right).

In particular, \eqref{eq3.7} is true for the rightmost backwards path starting at $B_-$ with step initial condition $h^{\rm step}_{-\delta t/2,0}(\cdot,0).$
 Hence, on the event $G$, \eqref{eq3.7} is true for $ \x^{-}_2$. By assumption (b), the probability of $G$ converges to $1$, hence \eqref{eq3.7} is true for $ \x^{-}_2$ unconditionally. Since $ \x^{-}_2$ is the rightmost backwards path, \eqref{eq3.7} then applies to all backwards path.

  \end{proof}

\subsection{Slow decorrelation}
This result is an analogue to Theorem 3.2 of~\cite{CFP10b}, specialized to our setting.
\begin{thm}\label{thmSlowDec}
Let $v\in\R$, $w\in\R$ and $\tilde\nu<1$ such that
\begin{equation}
\begin{aligned}
&\frac{h(v T,T)-\mu_1(T)}{T^{1/3}}\Rightarrow D,\textrm{ as }T\to\infty,\\
&\frac{h(v T-w T^{\tilde \nu},T-T^{\tilde \nu})-\mu_2(T)}{T^{1/3}}\Rightarrow D,\textrm{ as }T\to\infty,\\
&\frac{h^{\rm step}_{v T-w T^{\tilde \nu},T-T^{\tilde \nu}}(v T,T)-(\mu_1(T)-\mu_2(T))}{T^{1/3}}\Rightarrow 0,\textrm{ as }T\to\infty,
\end{aligned}
\end{equation}
for some distribution $D$ and some $\mu_1(T),\mu_2(T)$. Then, for any $\e>0$,
\begin{equation}
\lim_{T\to\infty}\Pb(|h(v T,T)-h(v T-w T^{\tilde \nu},T-T^{\tilde \nu})-(\mu_1(T)-\mu_2(T))|\geq \e T^{1/3})=0.
\end{equation}
\end{thm}

We apply slow decorrelation to the setting of Figure~\ref{FigSlowDec}.
\begin{figure}
\begin{center}
\includegraphics[height=5cm]{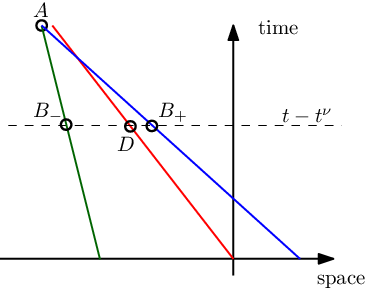}
\caption{The setting for the application of slow decorrelation. The red line is the macroscopic trajectory of the shock.}
\label{FigSlowDec}
\end{center}
\end{figure}
Assumption (c) tells us that the limiting distributions of the scaled height functions at space-time points $A$ and $B_-$ are the same. Slow decorrelation implies that they are asymptotically very correlated since the probability of their difference goes to zero, see Proposition~\ref{propSlowDecminus}.
\begin{prop}\label{propSlowDecminus}Let $\nu\in(2/3,1)$, $A$ and $B_-$ as in \eqref{eqAB}. Then, for any $\e>0$,
\begin{equation}
\Pb\left(\frac{|h^-(A)-h^-(B_-)-((1-2\chi_-)(\tau t^{2/3}+t^\nu)+(1-2\lambda) s t^{1/3})|}{t^{1/3}}\geq \e\right)\to 0
\end{equation}
as $t\to\infty$.
\end{prop}
\begin{proof}
We apply Theorem~\ref{thmSlowDec} with $T=t+\tau t^{2/3}$, $v=v_s+s t^{1/3}/T$, $T^{\tilde\nu}=t^\nu+\tau t^{2/3}$ and $w=1-2\lambda+s t^{1/3}/(\tau t^{2/3}+t^\nu)$. Let us verify that the assumptions are satisfied.

(a) By Assumption (c) with $\theta=\tau t^{2/3}$ and $u=(\lambda-\rho)\tau+s t^{-1/3}$ we get
\begin{equation}\label{eq3.6}
\lim_{t\to\infty}\frac{h^-(A)-H_A}{-t^{1/3}}={\cal H}^-((\lambda-\rho)\tau),
\end{equation}
where $H_A=\mu_s (t+\tau t^{2/3})+(1-2\lambda)s t^{1/3}$.

(b) By Assumption (c) with $\theta=-t^\nu$ and $u=(\lambda-\rho)\tau$ we get
\begin{equation}\label{eq2.9}
\lim_{t\to\infty}\frac{h^-(B_-)-H_{B_-}}{-t^{1/3}}={\cal H}^-((\lambda-\rho)\tau),
\end{equation}
 where $H_{B_-}=\mu_s t-(1-2\lambda)(\rho-\lambda)\tau t^{2/3}-(1-2\chi_-)t^\nu$.

(c) Let $C_-=A-B_-=((1-2\lambda)(\tau t^{2/3}+t^\nu)+s t^{1/3},\tau t^{2/3}+t^\nu)$, then
\begin{equation}
\lim_{t\to\infty}\frac{h^{\rm step}(C_-)-(1-2\chi_-)(\tau t^{2/3}+t^\nu)-(1-2\lambda) s t^{1/3}}{-t^{1/3}} =0,
\end{equation}
since if we divide by $t^{\nu/3}$ is converges to a (scaled) GUE Tracy-Widom distribution function and $t^{\nu/3-1/3}\to 0$ as $t\to\infty$.

Finally one verifies the identity
\begin{equation}
H_A=H_{B_-}+(1-2\chi_-)(\tau t^{2/3}+t^\nu)+(1-2\lambda) s t^{1/3},
\end{equation}
which then ensures that we can apply Theorem~\ref{thmSlowDec}.
\end{proof}

\begin{prop}\label{propSlowDecplus}
Let $\nu\in(2/3,1)$, $A$ and $B_+$ as in \eqref{eqAB}. Then, for any $\e>0$,
\begin{equation}
\Pb\bigg(\frac{|\tilde h^+(A)-\tilde h^+(B_+)-((1-2\chi_+)(\tau t^{2/3}+t^\nu)+(1-2\rho) s t^{1/3})|}{t^{1/3}}\geq \e\bigg)\to 0
\end{equation}
as $t\to\infty$.
\end{prop}
\begin{proof}
The proof is as for Proposition~\ref{propSlowDecminus}, except that now we have
\begin{equation}\label{eq2.13}
\lim_{t\to\infty}\frac{\tilde h^+(B_+)-H_{B_+}}{-t^{1/3}}={\cal H}^+((\rho-\lambda)\tau),
\end{equation}
with $H_{B_+}=\mu_s t+(1-2\rho)(\rho-\lambda)\tau t^{2/3}-(1-2\chi_+)t^\nu$ as well as
and, with $C_+=A-B_+=((1-2\rho)(\tau t^{2/3}+t^\nu)+s t^{1/3},\tau t^{2/3}+t^\nu)$,
\begin{equation}
\lim_{t\to\infty}\frac{h^{\rm step}(C_+)-(1-2\chi_+)(\tau t^{2/3}+t^\nu)-(1-2\rho) s t^{1/3}}{-t^{1/3}} =0.
\end{equation}
\end{proof}

\subsection{Proof of Theorem~\ref{thmMain}}
Since we need to deal with different $\tau_k$ jointly, let us make the dependance on $\tau$ clear in the notation: We rewrite the points $A,B , B_+ $ (defined earlier in \eqref{eqAB}) as
\begin{equation}\label{eqABbbis}
\begin{aligned}
A^{k}&=(v_s(t+\tau_{k} t^{2/3})+s t^{1/3},t+\tau_{k} t^{2/3}),\\
B^{k}_{-}&=(v_s (t+\tau_{k} t^{2/3})-(1-2\lambda)(\tau_{k} t^{2/3}+t^\nu),t-t^\nu),\\
B^{k}_{+}&=(v_s (t+\tau_{k} t^{2/3})-(1-2\rho)(\tau_{k} t^{2/3}+t^\nu),t-t^\nu).
\end{aligned}
\end{equation}
By Theorem~\ref{thmDistrSecondClass} we have
\begin{equation}\label{eq4.8}
\begin{aligned}
&\Pb\left(\bigcap_{k=1}^{m}\{X^{\rm 2nd}(t+\tau_{k} t^{2/3})\geq v_s (t+\tau_{k} t^{2/3})+s t^{1/3}\}\right)\\
&=\Pb\left(\bigcap_{k=1}^{m}\{h^-(A^{k})-\mu_s(t+\tau_{k} t^{2/3}) \leq \tilde h^+(A^{k})-\mu_s(t+\tau_{k} t^{2/3})\}\right).
\end{aligned}
\end{equation}
Next we use slow decorrelation. This allows us to replace the point $A^{k}$ in \eqref{eq4.8} by $B^{k}_{-}$ (for $h^{-}$) and resp. by $B^{k}_{+}$ (for $ \tilde h^+$) upon adapting the rescaling.
We define
\begin{equation}
\begin{aligned}\label{Z}
&Z^{-}(B^{k}_{-})=\frac{h^-(B^{k}_{-})-\mu_s (t+\tau_{k} t^{2/3})+(1-2\chi_-)(\tau_{k} t^{2/3}+t^{\nu})+(1-2\lambda)s t^{1/3}}{-t^{1/3}},\\&
Z^{+}(B^{k}_{+})= \frac{\tilde h^+(B^{k}_{+})-\mu_s(t+\tau_{k} t^{2/3})+(1-2\chi_+)(\tau_{k} t^{2/3}+t^\nu)+(1-2\rho)s t^{1/3}}{-t^{1/3}}.
\end{aligned}
\end{equation}
By Propositions~\ref{propSlowDecminus} and~\ref{propSlowDecplus}, as $t\to\infty$, we have
\begin{equation}\label{eq4.9}
\begin{aligned}
&\lim_{t\to \infty} \Pb\left(\bigcap_{k=1}^{m}\{h^-(A^{k})-\mu_s(t+\tau_{k} t^{2/3}) \leq \tilde h^+(A^{k})-\mu_s(t+\tau_{k} t^{2/3})\}\right)\\&=\lim_{t\to\infty}\Pb\left(\bigcap_{k=1}^{m}\{Z^{-}(B^{k}_{-})\leq Z^{+}(B^{k}_{+})\}\right).
\end{aligned}
\end{equation}
Define now the space-time regions
\begin{equation}
\begin{aligned}
\mathcal{D}^{-}=\{(x,\ell):x<v_s \ell\}, \quad \mathcal{D}^{+}=\{(x,\ell):x>v_s \ell\}.
\end{aligned}
\end{equation}
Also define the height functions $h^{-}_{\mathcal{D}^{-}}(B^{k}_{-})$ (resp.\ $\tilde h^{+}_{\mathcal{D}^{+}}(B_{+}^k)$) which differ from $h^{-}(B^{k}_{-})$ (resp.\ $\tilde h^{+}(B_{+}^k)$) in that the former the height function to right of $\mathcal{D}^{-}$ (resp.\ to the left of $\mathcal{D}^{+}$) has a deterministic fixed slope $1$ (resp.\ $-1$). We define random variables $Z_{\mathcal{D}^{-}}^{-}(B^{k}_{-})$ and $Z_{\mathcal{D}^{+}}^{+}(B^{k}_{+})$ from $h^{-}_{\mathcal{D}^{-}}(B^{k}_{-})$ and $\tilde h^{+}_{\mathcal{D}^{+}}(B_{+}^k)$ as in \eqref{Z}.

Since $\mathcal{D}^{-}$ and $\mathcal{D}^{+}$ are disjoint, $(Z_{\mathcal{D}^{-}}^{-}(B^{k}_{-}),1\leq k\leq m)$ and $(Z_{\mathcal{D}^{+}}^{+}(B^{k}_{+}),1\leq k\leq m)$ are independent random vectors. Let $\x^{+}_{k}$ and $\x^{-}_{k}$ be the geodesics starting from $B_{-}^{k},B_{+}^{k}$ as defined in Proposition~\ref{propLocalization}. Furthermore, we define the events
\begin{equation}
E^{+}_{k}=\{(\x^{+}_{k}(\ell),\ell)\in\mathcal{D}^{+},\ell\leq t-t^{\nu}\},\quad E^{-}_k=\{(\x^{-}_{k}(\ell),\ell)\in\mathcal{D}^{-},\ell \leq t-t^{\nu}\}.
\end{equation}
Then on $E^{+}_k\cap E^{-}_{k}$ we have $Z_{\mathcal{D}^{\pm}}^{\pm}(B_\pm^{k})=Z^\pm (B_\pm^{k})$. Proposition~\ref{propLocalization} gives
\begin{equation}
\lim_{t\to\infty}\Pb\left(\bigcap_{k=1}^{m}E^{+}_{k}\cap \bigcap_{k=1}^{m}E^{-}_{k} \right)=1.
\end{equation}
We thus see that the convergence
 \eqref{eq2.9} and \eqref{eq2.13} to $\mathcal{H}^{\pm}$  applies to
$Z_{\mathcal{D}^{\pm}}^{\pm}(B_\pm^{k})$ as well.
 Thus we get, as $t\to\infty$,
\begin{equation}\label{eq3.33}
\begin{aligned}
&\lim_{t\to\infty}\Pb\left(\bigcap_{k=1}^{m}\{Z^{-}(B^{k}_{-})\leq Z^{+}(B^{k}_{+})\}\right)=\lim_{t\to\infty}\Pb\left(\bigcap_{k=1}^{m}\{Z_{\mathcal{D}^{-}}^{-}(B^{k}_{-})\leq Z_{\mathcal{D}^{+}}^{+}(B^{k}_{+})\}\right)\\
&= \Pb\left(\bigcap_{k=1}^{m}\{(1-2\lambda)s-{\cal H}^-((\lambda-\rho)\tau_{k})\leq(1-2\rho)s- {\cal H}^+((\rho-\lambda)\tau_{k})\}\right),
\end{aligned}
\end{equation}
where the two limit processes ${\cal H}^-$ and ${\cal H}^+$ are independent. Reordering the terms in \eqref{eq3.33} leads to \eqref{eq3.8b}.

\subsection{Proof of Corollary~\ref{truecor}}
What we need is to verify Assumptions (a)-(c) of Section~\ref{sectAsymptotics} and identify the limit in (c).

Recall that the initial condition is \eqref{eq5.5} given by
\begin{equation}
x_n(0)=\left\{\begin{array}{ll}
-\lfloor n/\lambda\rfloor, &n\geq 1,\\
-\lfloor n/\rho\rfloor,& n\leq 0.
\end{array}\right.
\end{equation}
Then Assumption (a) is obvious.

Assumption (b) on the control of the end-point of the backwards geodesics is the following.
\begin{prop}\label{prop5.3}
Let $\x^-$ (resp.\ $\x^+$) be backwards geodesics of $h^-$ (resp.\ $\tilde h^+$) starting from $v_s t$ at time $t$. Then, for $\delta=\frac12(\rho-\lambda)$ we have
\begin{equation}
\lim_{t\to\infty}\Pb(\x^-(0)\leq -\delta t,\x^+(0)\geq \delta t) = 1.
\end{equation}
\end{prop}
\begin{proof}
Let us prove the statement for $\x^-(0)$. The proof for $\x^+(0)$ is completely analogous.

We expect that the characteristics from $((1-\lambda-\rho)t,t)$ will be around $(\lambda-\rho)t$ at time $0$. So the step initial condition from that point is going to be a relatively good upper bound for $h^-$, namely
\begin{equation}
h^-(v_s t,t)\leq h^-((\lambda-\rho)t,0)+h^{\rm step}_{(\lambda-\rho)t,0}(v_s t,t).
\end{equation}
Define the height function arising from the minimization of \eqref{minid} from $-\delta t$ to $0$ as
\begin{equation}
h^{\rm ld}(v_s t,t):=\min_{-\delta t< z\leq 0}\{h^-(z,0)+h^{\rm step}_{z,0}(v_s t,t)\}.
\end{equation}
Then, for  $S$ arbitrary we have
\begin{equation}\label{eq5.13}
\begin{aligned}
&\Pb(\x^-(0)>-\delta t)\leq \Pb\big(h^{\rm ld}(v_s t,t) \leq  h^-((\lambda-\rho)t,0)+h^{\rm step}_{(\lambda-\rho)t,0}(v_s t,t)\big)\\
&\leq \sum_{-\delta t< z\leq 0} \Pb\big(h^-(z,0)+h^{\rm step}_{z,0}(v_s t,t) \leq  h^-((\lambda-\rho)t,0)+h^{\rm step}_{(\lambda-\rho)t,0}(v_s t,t)\big)\\
&\leq \sum_{-\delta t< z\leq 0} \Pb\big(h^-(z,0)+h^{\rm step}_{z,0}(v_s t,t) \leq  S\big)\\
&+\sum_{-\delta t<z\leq 0}\Pb\big(S\leq h^-((\lambda-\rho)t,0)+h^{\rm step}_{(\lambda-\rho)t,0}(v_s t,t)\big).
\end{aligned}
\end{equation}
With the initial conditions \eqref{eq5.5} we have $h^-(z,0)=(1-2\lambda)z+\Or(1)$. Also, see \eqref{eqBoundStepIC},
\begin{equation}
\begin{aligned}
&h^-(-\beta t,0)+h^{\rm step}_{-\beta t,0}(v_s t,t) = (1-\lambda-\rho+2\lambda\rho)t + \tfrac12(\beta+\lambda-\rho)^2 t+\Or_p (t^{1/3}),\\
&h^-((\lambda-\rho)t,0)+h^{\rm step}_{(\lambda-\rho)t,0}(v_s t,t) =(1-\lambda-\rho+2\lambda\rho)t+\Or_p (t^{1/3}).
\end{aligned}
\end{equation}
Since $\frac12(\beta+\lambda-\rho)^2\geq \frac18 (\rho-\lambda)^2$ for all $\beta\in [0,\delta]$, we set in \eqref{eq5.13} the value $S=(1-\lambda-\rho+2\lambda\rho)t + \frac{1}{16}(\rho-\lambda)^2 t$. From the upper and lower bound of $h^{\rm step}$, see \eqref{eqBoundStepIC}, we get that each of the summand in the r.h.s.\ of \eqref{eq5.13} is bounded by $C e^{-c t^{2/3}}$ for some constants $C,c>0$. As we have only $\Or(t)$ summands, the proof is completed.
\end{proof}

Finally, Assumption (c) is a special case of Corollary~\ref{corAfinal} as the following result shows.
\begin{cor}\label{corAspecial}
Let $v_s=1-\lambda-\rho$, $\chi_-=\lambda(1-\lambda)$, $\chi_+=\rho(1-\rho)$, and ${\bf h}^\pm(u,\theta,t)$ as in \eqref{eq2.4}. Then
\begin{equation}
\begin{aligned}
&\lim_{t\to\infty} \frac{h^-(v_s t+u t^{2/3}+(1-2\lambda)\theta,t+\theta)-{\bf h}^-(u,\theta,t)}{-t^{1/3}},\\
&=\chi_-^{2/3} 2^{4/3} {\cal A}_1(\chi_-^{-1/3}2^{-5/3}u)
\end{aligned}
\end{equation}
and
\begin{equation}
\begin{aligned}
&\lim_{t\to\infty}\frac{\tilde h^+(v_s t+u t^{2/3}+(1-2\rho)\theta,t+\theta)-{\bf h}^+(u,\theta,t)}{-t^{1/3}}\\
&=\chi_+^{2/3} 2^{4/3} \widetilde {\cal A}_1(\chi_+^{-1/3}2^{-5/3}u).
\end{aligned}
\end{equation}
\end{cor}
\begin{proof}
The result for $h^-$ follows from Corollary~\ref{corAfinal}, see reformulation \eqref{eqA3}, by setting $t\to t+\theta$ and $a=\frac{(\rho-\lambda)t-u t^{2/3}+u(t+\theta)^{2/3}}{t+\theta}$. Since $\rho-\lambda>0$, $a\to \rho-\lambda>0$ as $t\to\infty$.

The result for $\tilde h^+$ can be obtained by employing particle-hole symmetry, since the evolution of the height function can be equivalently described in terms of the jumps of the holes (when they jump to the left, the height increases by $2$) and the initial condition of the holes will have density $\rho$ on $\Z_+$ and empty on the left of the origin.
\end{proof}

\appendix

\section{Limit process for flat IC}\label{AppFlat}
Let us consider TASEP described in terms of particle positions $X_t(n)$ ordered from right to left with initial condition
\begin{equation}
X_0(n)=-\lfloor n/\lambda\rfloor,\quad n\geq 1.
\end{equation}
Let $h(x,t)$ be the associated height function, which can be expressed as
\begin{equation}
h(x,t)=2 (X^{-1}_t(x-1)-X^{-1}_0(-1))+x,
\end{equation}
where $X^{-1}_t(x)=\min\{k\in\Z\, |\, X_t(k)\leq x\}$. With our initial condition we have as in~\cite{MQR17}, $X^{-1}_0(-1)=1$ that is, the particle with label $1$ is initially rightmost in $\Z_-=\{\ldots,-2,-1\}$. The only difference is that our height function is $-1$ times the height function of~\cite{MQR17}. In the following, we adopt the notations of~\cite{MQR17}.

We want to show that, setting $\chi=\lambda(1-\lambda)$ and any $a>0$, we have
\begin{equation}\label{eqA3}
\begin{aligned}
&\lim_{t\to\infty}\frac{h^-((1-2\lambda)t-a+u t^{2/3},t)-(1-2\chi_-)t+(1-2\lambda)(at-u t^{2/3})}{-t^{1/3}}\\
&=\chi^{2/3}2^{4/3} {\cal A}_1(\chi^{-1/3}2^{-5/3} u)
\end{aligned}
\end{equation}
in the sense of finite-dimensional convergence. The one-point convergence for general density $\lambda$ was proven in~\cite{FO17}.

Since ${\cal A}_1(u)\stackrel{(d)}{=}{\cal A}_1(-u)$, this statement is equivalent to showing the convergence of
\begin{equation}\label{eqA1}
h^{\rm resc}_t(\xi)=\frac{h((1-2\lambda)t-a(\xi,t),t)-(1-2\chi)t+(1-2\lambda)a(\xi,t)}{-2^{4/3}\chi^{2/3} t^{1/3}}\to {\cal A}_1(\xi)
\end{equation}
as $t\to\infty$, where we used the notation $a(\xi,t)=at+2^{5/3}\chi^{1/3}\xi t^{2/3}$.

Some elementary algebra leads to\footnote{We have done some shift by $o(t^{-1/3})$ to $\xi_i$ and $s_i$ to have $n_i$ independent of $s_i$, but in the limit is the same since the limit is continuous in $\xi_i$ and $s_i$ and the differences goes to $0$ as $t\to\infty$.}
\begin{equation}
\lim_{t\to\infty}\Pb(h^{\rm resc}_t(\xi_i)\leq s_i,1\leq i \leq m) = \lim_{t\to\infty}\Pb(X_t(n_i)>a_i,1\leq i\leq m)
\end{equation}
where
\begin{equation}\label{eqA4}
\begin{aligned}
n_i&=\lambda^2 t+a\lambda t+\lambda 2^{5/3}\chi^{1/3}\xi_i t^{2/3},\\
a_i&=(1-2\lambda)t-a t-2^{5/3}\chi^{1/3}\xi_i t^{2/3}-2^{1/3}\chi^{2/3} \frac{s_i}{\lambda} t^{1/3}.
\end{aligned}
\end{equation}
We start considering $a=0$, where under the above scaling we should see the Airy$_{2\to 1}$ transition process.

For $\lambda=1/2$ the rigorous asymptotic was made in~\cite{BFS07} using the explicit representation of the kernel. For $\lambda=1/d$, with $d=2,3,\ldots$ the analysis was made in~\cite{BFP06} (in a discrete time version). For more general initial condition with average density $1/2$, this was carried out in~\cite{MQR17}. The approach of~\cite{MQR17} is robust and can be applied also to generic average density. Below we indicate the main steps focusing on the differences with respect to $\lambda=1/2$. We however will not write down all details of the asymptotic analysis of the appearing kernel and the (exponential) bounds allowing to exchange limits and sums in the expression of the kernel as well as to get the convergence of the Fredholm determinants. The strategy for the analysis of contour integrals are present in many papers, including the above mentioned ones or also Section~6.1 of~\cite{BF08}. The asymptotic for $\lambda\neq 1/2$ does not introduce any new technical difficulties as we could check. For instance the question of trace-class is carefully discussed in Appendix~B of~\cite{MQR17}.

We start with Theorem~2.6 of~\cite{MQR17} (which holds for right-finite initial data as in our case), where they get an explicit enough expression for the biorthogonalization in~\cite{BFPS06}.
\begin{thm}[Theorem~2.6 of~\cite{MQR17}]\label{thmA1} We have
\begin{equation}
\Pb(X_t(n_i)>a_i,1\leq i\leq m) = \det(\Id-\bar\chi_a K_t \bar\chi_a)_{\ell^2(\{n_1,\ldots,n_m\}\times\Z)},
\end{equation}
where $\bar\chi_a(n_i,x)=\Id_{x\leq a_i}$,
\begin{equation}
K_t(n_i,x_i;n_j,x_j)=-Q^{n_j-n_i}(x_i,x_j)\Id_{n_i<n_j} + \sum_{y} S_{-t,-n_i}(y,x_i) \bar S^{{\rm epi}(X_0)}_{-t,n_j}(y,x_j)
\end{equation}
with
\begin{equation}\label{eqA5}
\begin{aligned}
Q^{n_j-n_i}(x_i,x_j)&=\frac{1}{2^{x_i-x_j}} \binom{x_i-x_j-1}{n_j-n_i-1}\Id_{x_i-x_j\geq n_j-n_i},\\
S_{-t,-n_i}(y,x_i)&=\frac{1}{2^{x_i-y}}\frac{1}{2\pi\I}\oint_{\Gamma_0}dw \frac{(1-w)^{n_i} e^{t(w-1/2)}}{w^{n_i+1+x_i-y}},\\
\bar S_{-t,n_j}(y,x_j)&=\frac{1}{2^{y-x_j}}\frac{1}{2\pi\I}\oint_{\Gamma_0} dw \frac{(1-w)^{x_j-y+n_j-1}e^{t(w-1/2)}}{w^{n_j}},\\
\bar S^{{\rm epi}(X_0)}_{-t,n_j}(y,x_j)&=\E_{B_0=y}\left[\bar S_{-t,n_j-\tau}(B_\tau,x_j)\Id_{\tau<n_j}\right].
\end{aligned}
\end{equation}
Here $\tau=\min\{m\geq 0\, |\, B_m>X_0(m+1)\}$, where $B_m$ is a random walk with transition matrix $Q$, that is, $\Pb(B_m-B_{m-1}=-k)=2^{-k}$, $k\geq 1$.
\end{thm}

\begin{remark}
If one tries to apply directly Theorem~\ref{thmA1} with the scaling \eqref{eqA4} with $a>0$, one gets in trouble. The reason is that the random walk $B_m$ starts from $0$, which is however not the macroscopic position of the ending of the backwards geodesics starting at $(a_i,t)$. For this case one should rather look for a two-sided finite formula in the spirit of Section~3.4 of~\cite{MQR17}. We will rather argue that the limit process for $a>0$ is the same as the one obtained for $a=0$ with a subsequent change of variables $\xi_i\to\xi_i+L$ in the $L\to\infty$ limit.
\end{remark}

Consider the following scaling for $x_i$ and $y$:
\begin{equation}\label{eqA6}
\begin{aligned}
n_i&=\lambda^2 t+\lambda 2^{5/3}\chi^{1/3}\xi_i t^{2/3},\\
x_i&=(1-2\lambda)t-2^{5/3}\chi^{1/3}\xi_i t^{2/3}-2^{1/3}\chi^{2/3} \frac{u_i}{\lambda}t^{1/3},\\
y&=2^{1/3}\chi^{2/3} \frac{v}{\lambda}t^{1/3}.
\end{aligned}
\end{equation}

The critical point for the steep descent analysis of $S_{-t,-n_i}(y,x_i)$ is $1-\lambda$ and the one for $\bar S_{-t,n_j}(y,x_j)$ is $\lambda$. For that reason we conjugate the entries of the kernel $K_t$ to avoid dealing with it in the asymptotics, and multiply by $e^{-t(1/2-\lambda)}$ (resp.\ $e^{t(1/2-\lambda)}$) the term $S_{-t,-n_i}(y,x_i)$ (resp.\ $\bar S_{-t,n_j}(y,x_j)$). Thus, Theorem~\ref{thmA1} holds true if we redefine the entries of \eqref{eqA5} as follows:
\begin{equation}\label{eqA7}
\begin{aligned}
Q^{n_j-n_i}(x_i,x_j)&=\frac{(1-\lambda)^{n_i-n_j+x_i-x_j}}{\lambda^{n_i-n_j}} \binom{x_i-x_j-1}{n_j-n_i-1}\Id_{x_i-x_j\geq n_j-n_i},\\
S_{-t,-n_i}(y,x_i)&=\frac{(1-\lambda)^{n_i+x_i-y}}{\lambda^{n_i}}\frac{1}{2\pi\I}\oint_{\Gamma_0}dw \frac{(1-w)^{n_i} e^{t(w-1+\lambda)}}{w^{n_i+1+x_i-y}},\\
\bar S_{-t,n_j}(y,x_j)&=\frac{(1-\lambda)^{-n_j+y-x_j}}{\lambda^{-n_j}} \frac{1}{2\pi\I}\oint_{\Gamma_0} dw \frac{(1-w)^{x_j-y+n_j-1}e^{t(w-\lambda)}}{w^{n_j}},\\
\bar S^{{\rm epi}(X_0)}_{-t,n_j}(y,x_j)&=\E_{B_0=y}\left[\bar S_{-t,n_j-\tau}(B_\tau,x_j)\Id_{\tau<n_j}\right].
\end{aligned}
\end{equation}
From now we are going to use \eqref{eqA7} instead of \eqref{eqA5}.

\begin{lem}\label{lemAQ}
Under the scaling \eqref{eqA6} we get, for $\xi_j>\xi_i$,
\begin{equation}
\lim_{t\to\infty} \frac{2^{1/3}\chi^{2/3} t^{1/3}}{\lambda} Q^{n_j-n_i}(x_i,x_j)= \frac{1}{\sqrt{4\pi (\xi_j-\xi_i)}} e^{-(u_j-u_i)^2/(4 (\xi_j-\xi_i))}.
\end{equation}
\end{lem}
\begin{proof}
This can be obtained using the Stirling formula for the factorials or doing asymptotic analysis in an integral representation of the binomial coefficients.
\end{proof}

\begin{lem}\label{lemA1}
Under the scaling \eqref{eqA6} we have:
\begin{equation}
\begin{aligned}
\lim_{t\to\infty} &\frac{2^{1/3}\chi^{2/3} t^{1/3}}{\lambda} S_{-t,n_i}(y,x_i) = 2^{1/3}\frac{1}{2\pi\I}\int_{\langle} dW e^{W^3/3-2^{2/3}\xi_i W^2-2^{1/3}(u_i+v)W}\\
&=2^{1/3} \Ai(2^{4/3}\xi_i^2+2^{1/3}(u_i+v)) e^{-\frac83 \xi_i^3-2\xi_i(u_i+v)}.
\end{aligned}
\end{equation}
\end{lem}
\begin{proof}
The proof follow the standard steep descent analysis by inserting the scaling \eqref{eqA4} and \eqref{eqA6}, and extracting the leading contribution around the critical point $1-\lambda$. In the computation we do the change of variable $w=1-\lambda-\chi^{1/3} t^{-1/3} W$. The expression in terms of the Airy function follows for instance by (A.1) of~\cite{BFP09}.
\end{proof}

A standard application of Donsker's theorem leads to the convergence of $B_m$ to a Brownian motion. More precisely, consider the scaling , then
\begin{equation}\label{eqA8}
\begin{aligned}
\tau&=2^{5/3}\chi^{1/3}\lambda t^{2/3} T,\\
B_\tau&= -2^{5/3}\chi^{1/3}t^{2/3} T +2^{1/3}\chi^{2/3} \frac{\beta_T}{\lambda} t^{1/3}.
\end{aligned}
\end{equation}
Then as $t\to\infty$, where $\beta_T$ converges to a Brownian motion with diffusivity constant $2$.
\begin{lem}Under the scaling \eqref{eqA6} and \eqref{eqA8} we have:
\begin{equation}
\begin{aligned}
\lim_{t\to\infty} &\frac{2^{1/3}\chi^{2/3} t^{1/3}}{\lambda} \bar S_{-t,n_j-\tau}(B_\tau,x_j)\\& = 2^{1/3}\frac{1}{2\pi\I}\int_{\langle} dW e^{W^3/3+2^{2/3}(\xi_j-T) W^2-2^{1/3}(u_j+\beta_T)W}\\
&=2^{1/3} \Ai(2^{4/3}(\xi_j-T)^2+2^{1/3}(u_j+\beta_T)) e^{\frac83 (\xi_j-T)^3+2(\xi_j-T)(u_j+\beta_T)}.
\end{aligned}
\end{equation}
\end{lem}
\begin{proof}
This time we do the change of variables around the critical point $\lambda$ as $w=\lambda-\chi^{1/3} t^{-1/3} W$.
\end{proof}

From this we get the limit of $\bar S^{{\rm epi}(X_0)}$.
\begin{lem}\label{lem3A}
Under the scaling \eqref{eqA6} we have:
\begin{equation}
\begin{aligned}
\lim_{t\to\infty} &\frac{2^{1/3}\chi^{2/3} t^{1/3}}{\lambda} \bar S^{{\rm epi}(X_0)}_{-t,n_j}(y,x_j)\\
 =&2^{1/3} \Ai(2^{4/3}\xi_j^2+2^{1/3}(u_j+v)) e^{\frac83 \xi_i^3+2\xi_j(u_j+v)} \Id_{v\geq 0}\\
 &+ 2^{1/3} \Ai(2^{4/3} \xi_j^2 +2^{1/3}(u_j-v)) e^{\frac83 \xi_j^3+2\xi_j(u_j-v)} \Id_{v<0}.
\end{aligned}
\end{equation}
\end{lem}
\begin{proof}We start with
\begin{equation}
\frac{2^{1/3}\chi^{2/3} t^{1/3}}{\lambda} \bar S^{{\rm epi}(X_0)}_{-t,n_j}(y,x_j) = \lim_{t\to\infty}\int_{\R_+} d\nu_v(T) \frac{2^{1/3}\chi^{2/3} t^{1/3}}{\lambda}\bar S^{{\rm epi}(X_0)}_{-t,n_j-\tau}(B_\tau,x_j),
\end{equation}
where $d\nu_v(T)$ is the law of the hitting time $T=\inf\{t\geq 0\,|\, \beta_T\geq 0\}$ with $\beta_0=v$.

Case 1: $v\geq 0$. Then $d\nu_v(T)=\delta_0(T)$ and $\beta_0=v$. Then
\begin{equation}
\lim_{t\to\infty} \frac{2^{1/3}\chi^{2/3} t^{1/3}}{\lambda} \bar S^{{\rm epi}(X_0)}_{-t,n_j}(y,x_j)=2^{1/3} \Ai(2^{4/3}\xi_j^2+2^{1/3}(u_j+v)) e^{\frac83 \xi_i^3+2\xi_j(u_j+v)}
\end{equation}

Case 2: $v<0$. As $t\to\infty$ we have $d\nu_v(T)=\frac{-v}{\sqrt{4\pi T^3}} e^{-v^2/(4T)} dT$ and $\beta_T=0$. This leads to
\begin{equation}
\begin{aligned}
\lim_{t\to\infty} &\frac{2^{1/3}\chi^{2/3} t^{1/3}}{\lambda} \bar S^{{\rm epi}(X_0)}_{-t,n_j}(y,x_j)\\
&=\int_0^\infty dT \frac{-v}{\sqrt{4\pi T^3}} e^{-v^2/(4T)} \frac{2^{1/3}}{2\pi\I}\int dW e^{W^3/3-2^{1/3} u_j W+2^{2/3} (\xi_j-T)W^2}\\
&=2^{1/3} \Ai(2^{4/3} \xi_j^2 +2^{1/3}(u_j-v)) e^{\frac83 \xi_j^3+2\xi_j(u_j-v)},
\end{aligned}
\end{equation}
where first one does the integral over $T$ (it is explicit) and then we rewrite the result in terms of Airy functions and exponentials.
\end{proof}

Combining Lemma~\ref{lemA1} and Lemma~\ref{lem3A} (modulo some a priori bounds for large $|v|$ allowing to exchange the $t\to\infty$ limit and the sum over $y$, which can be obtained also with usual steep descent analysis) one obtains the following result.
\begin{prop}\label{propA6}
Under the scaling \eqref{eqA6},
\begin{equation}
\begin{aligned}
\lim_{t\to\infty}&\frac{2^{1/3}\chi^{2/3} t^{1/3}}{\lambda}  \sum_{y} S_{-t,-n_i}(y,x_i) \bar S^{{\rm epi}(X_0)}_{-t,n_j}(y,x_j)\\
=& 2^{1/3}\frac{e^{\frac23 \tilde \xi_j^3+\tilde\xi_j \tilde u_j}}{e^{\frac23 \tilde\xi_i^3+\tilde\xi_i \tilde u_i}} \int_0^\infty dv e^{v(\tilde \xi_j-\tilde \xi_i)}\Ai(\tilde \xi_i^2+\tilde u_i+v) \Ai(\tilde \xi_j^2+\tilde u_j+v) \\
&+2^{1/3}\frac{e^{\frac23 \tilde \xi_j^3+\tilde\xi_j \tilde u_j}}{e^{\frac23 \tilde \xi_i^3+\tilde\xi_i \tilde u_i}} \int_{-\infty}^0 dv e^{-v(\tilde \xi_j+\tilde \xi_i)} \Ai(\tilde \xi_i^2+\tilde u_i+v) \Ai(\tilde\xi_j^2+\tilde u_j-v),
\end{aligned}
\end{equation}
where the last term can also be rewritten as
\begin{equation}
\begin{aligned}
&- 2^{1/3}\frac{e^{\frac23 \tilde \xi_j^3+\tilde \xi_j \tilde u_j}}{e^{\frac23 \tilde\xi_i^3+\tilde\xi_i \tilde u_i}} \int_0^\infty dv e^{-v(\tilde \xi_j+\tilde \xi_i)} \Ai(\tilde\xi_i^2+\tilde u_i+v) \Ai(\tilde\xi_j^2+\tilde u_j-v)\\
&+\Ai(u_i+u_j+(\xi_i-\xi_j)^2) e^{\frac23 (\xi_j-\xi_i)^3+(\xi_j-\xi_i)(u_i+u_j)},
\end{aligned}
\end{equation}
where we used the short cuts $\tilde\xi_i=2^{2/3}\xi_i$ and $\tilde u_i=2^{1/3}u_i$.
\end{prop}

Combining Lemma~\ref{lemAQ} and Proposition~\ref{propA6} we get the limit of the kernel under scaling \eqref{eqA6}.
\begin{prop}\label{propALimitkernel}
Under the scaling \eqref{eqA6} and $\tilde\xi_i=2^{2/3}\xi_i$ and $\tilde u_i=2^{1/3}u_i$,
\begin{equation}
\begin{aligned}
&\lim_{t\to\infty}\frac{2^{1/3}\chi^{2/3} t^{1/3}}{\lambda} K_t(n_i,x_i;n_j,x_j)\\
&= 2^{1/3} K_{2\to 1}(\tilde \xi_i,\tilde u_i+\min\{0,\tilde\xi_i\}^2;\tilde \xi_j,\tilde u_j+\min\{0,\tilde\xi_j\}^2),
\end{aligned}
\end{equation}
where $K_{2\to 1}$ is the Airy$_{2\to 1}$ kernel derived in~\cite{BFS07} (see Appendix~A of~\cite{BFS07} for explicit expressions)\footnote{To be precise, we obtain a conjugated version of the kernel written in~\cite{BFS07}.}.
\end{prop}

As mentioned above, we omit the proof of a-priori bounds for large $u_i$, which can be as usual obtained through steep descent analysis. These give the convergence of the Fredholm determinants as well, and not only of the kernel (see e.g.~Appendix~B of~\cite{MQR17} for  a discussion of the trace class convergence in a framework almost identical to the one considered here). The convergence of the Fredholm determinants implies the following result.
\begin{thm}\label{thmA9}
Consider the scaled height function \eqref{eqA1} with $a=0$. Then, for $\xi_1,\ldots,\xi_m\geq 0$,
\begin{equation}
\begin{aligned}
&\lim_{t\to\infty}\Pb(h^{\rm resc}_t(\xi_i)\leq s_i,1\leq i \leq m)\\
& = \Pb(2^{-1/3}{\cal A}_{2\to 1}(2^{2/3} \xi_i)-2\min\{0,\xi_i\}^2\leq s_i,1\leq i\leq m),
\end{aligned}
\end{equation}
where ${\cal A}_{2\to 1}$ is the Airy$_{2\to 1}$ process of~\cite{BFS07}.
\end{thm}

Finally, let us explain why the $a>0$ case can be recovered by the subsequent limits $t\to\infty$ and then shifting $\xi_i\to\infty$ simultaneously. Let us do the change of variables $\xi_i\to \xi_i+L$. Consider any backwards geodesics from ${\bf x}(t)=(1-2\lambda)t-2^{5/3}\chi^{1/3}\xi t^{2/3}$ at time $t$. On the event $\Omega_L=\{{\bf x}(0)<0\}$, the height function at time $t$ is not depending on the fact that the initial condition starts with empty initial condition on $\Z_+$. One can show that $\Pb(\Omega_L)\geq 1-C e^{-c L^2}$ for some constants $C,c>0$, thus $\Pb(\Omega_L)\to 1$ as $L\to\infty$. This was shown for general density $\lambda$ in Lemma~4.3 of~\cite{FO17} in the context of LPP, but it can be translated into this case as well. For $L\sim t^{\epsilon}$ for some $\epsilon>0$, the bound can be obtained in a similar way as in the proof of Proposition~\ref{prop5.3}.

This means that the law of $\lim_{L\to\infty}\lim_{t\to\infty}h^{\rm resc}_t(\xi_i+L)$ and the obtained by setting $a>0$ and taking $t\to\infty$ are the same. Since we have the identity $\lim_{L\to\infty}2^{-1/3} {\cal A}_{2\to 1}(2^{2/3}\xi+L)={\cal A}_1(\xi)$, see~\cite{BFS07}, we get the following corollary.
\begin{cor}\label{corAfinal}
Consider the scaled height function \eqref{eqA1} with $a>0$. Then, for $\xi_1,\ldots,\xi_m\geq 0$,
\begin{equation}
\lim_{t\to\infty}\Pb(h^{\rm resc}_t(\xi_i)\leq s_i,1\leq i \leq m) = \Pb({\cal A}_1(\xi_i)\leq s_i,1\leq i\leq m),
\end{equation}
where ${\cal A}_1$ is the Airy$_1$ process of~\cite{Sas05,BFPS06}.
\end{cor}

\section{A bound for step initial conditions}\label{AppStepIC}

Let $h^{\rm step}(x,0)=|x|$ and fixed $\alpha\in(-1,1)$. Then it is known that~\cite{Jo00b}
\begin{equation}
\lim_{t\to\infty} \Pb(h^{\rm step}(\alpha t,t)\leq \tfrac12(1+\alpha^2)t - s(1-\alpha^2)^{2/3} 2^{-1/3}t^{1/3})=F_{\rm GUE}(s),
\end{equation}
where $F_{\rm GUE}$ is the GUE Tracy-Widom distribution function. Furthermore, there exists constants $C,c>0$ such that for all $s\geq 0$
\begin{equation}\label{eqBoundStepIC}
\Pb(|h^{\rm step}(\alpha t,t)-\tfrac12(1+\alpha^2)t|\geq s t^{1/3}) \leq C e^{-c s}
\end{equation}
uniformly for all $t$ large enough. The constants can be chosen uniformly for $\alpha$ in a bounded set strictly away from $-1$ and $1$.
Using the relation with the Laguerre ensemble of random matrices (Proposition~6.1 of~\cite{BBP06}), or to TASEP, one sets the distribution is given by a Fredholm determinant.
An exponential decay of its kernel leads directly to the upper tail. See e.g.~Lemma~1 of~\cite{BFP09} for an explicit statement. The lower tail was proven in~\cite{BFP09}  (Proposition~3 together with (56)) with a better power $s^{3/2}$, that we do not need here.

\paragraph{Acknowledgements.} The work of P.L.~Ferrari is supported  by the Deutsche Forschungsgemeinschaft (German Research Foundation) by  the CRC 1060 (Projektnummer 211504053) and  Germany's Excellence Strategy - GZ 2047/1, Projekt ID 390685813.


\begin{thebibliography}{10}

\bibitem{Li85b}
T.~Liggett, Interacting Particle Systems, Springer Verlag, Berlin, 1985.

\bibitem{BGRS10}
C.~Bahadoran, H.~Guiol, K.~Ravishankar, E.~Saada, {Strong hydrodynamic limit
  for attractive particle systems on Z}, Electron. J. Probab. 15 (2010) 1--43.

\bibitem{Ev10}
L.~Evans, {Partial Differential Equations Second Edition}, Providence, RI,
  2010.

\bibitem{Li99}
T.~Liggett, Stochastic interacting systems: contact, voter and exclusion
  processes, Springer Verlag, Berlin, 1999.

\bibitem{Fer90}
P.~Ferrari, Shock fluctuations in asymmetric simple exclusion, Probab. Theory
  Relat. Fields 91 (1992) 81--101.

\bibitem{FF94b}
P.~Ferrari, L.~Fontes, {Shock fluctuations in the asymmetric simple exclusion
  process}, Probab. Theory Relat. Fields 99 (1994) 305--319.

\bibitem{PG90}
J.~G{\"a}rtner, E.~Presutti, Shock fluctuations in a particle system, Ann.
  Inst. H. Poincar{\'e} (A) 53 (1990) 1--14.

\bibitem{DMKPS89}
A.~D. Masi, C.~Kipnis, E.~Presutti, E.~Saada, Microscopic structure at the
  shock in the asymmetric simple exclusion, Stochastics and Stochastic Reports
  27 (1989) 151--165.

\bibitem{FGN17}
P.~Ferrari, P.~Ghosal, P.~Nejjar, {Limit law of a second class particle in
  TASEP with non-random initial condition}, Ann. Inst. Henri Poincar\'e Probab.
  Statist. 55 (2019) 1203--1225.

\bibitem{FN16}
P.~Ferrari, P.~Nejjar, {Fluctuations of the competition interface in presence
  of shocks}, ALEA, Lat. Am. J. Probab. Math. Stat. 14 (2017) 299--325.

\bibitem{N17}
P.~Nejjar, {Transition to shocks in TASEP and decoupling of last passage
  times}, ALEA, Lat. Am. J. Probab. Math. Stat. 15 (2018) 1311--1334.

\bibitem{FMP09}
P.~Ferrari, J.~Martin, L.~Pimentel, A phase transition for competition
  interfaces, Ann. Appl. Probab. 19 (2009) 281--317.

\bibitem{FN13}
P.~Ferrari, P.~Nejjar, {Anomalous shock fluctuations in TASEP and last passage
  percolation models}, Probab. Theory Relat. Fields 161 (2015) 61--109.

\bibitem{FN19}
P.~Ferrari, P.~Nejjar, {Statistics of TASEP with three merging
  characteristics}, J. Stat. Phys. 180 (2020) 398--413.

\bibitem{FN14}
P.~Ferrari, P.~Nejjar, {Shock fluctuations in flat TASEP under critical
  scaling}, J. Stat. Phys. 60 (2015) 985--1004.

\bibitem{QR18}
J.~Quastel, M.~Rahman, {TASEP fluctuations with soft-shock initial data},
  Annales Henri Lebesgue 3 (2020) 999--1021.

\bibitem{BB19}
A.~Borodin, A.~Bufetov, Color-position symmetry in interacting particle
  systems, Ann. Probab. 49 (2021) 1607--1632.

\bibitem{BF22}
A.~Bufetov, P.~Ferrari, {Shock fluctuations in TASEP under a variety of time
  scalings}, Ann. Appl. Probab. 32 (2022) 3614--3644.

\bibitem{N20CMP}
P.~Nejjar, {KPZ statistics of second class particles in ASEP via mixing},
  Commun. Math. Phys. 378 (2020) 601--623.

\bibitem{Sas05}
T.~Sasamoto, Spatial correlations of the {1D KPZ} surface on a flat substrate,
  J. Phys. A 38 (2005) L549--L556.

\bibitem{BFPS06}
A.~Borodin, P.~Ferrari, M.~Pr{\"a}hofer, T.~Sasamoto, {Fluctuation properties
  of the TASEP with periodic initial configuration}, J. Stat. Phys. 129 (2007)
  1055--1080.

\bibitem{MQR17}
K.~Matetski, J.~Quastel, D.~Remenik, {The KPZ fixed point}, Acta Math. 227
  (2021) 115--203.

\bibitem{FO17}
P.~Ferrari, A.~Occelli, {Universality of the GOE Tracy-Widom distribution for
  TASEP with arbitrary particle density}, Eletron. J. Probab. 23~(51) (2018)
  1--24.

\bibitem{CFP10b}
I.~Corwin, P.~Ferrari, S.~P{\'e}ch{\'e}, {Universality of slow decorrelation in
  KPZ models}, Ann. Inst. H. Poincar\'e Probab. Statist. 48 (2012) 134--150.

\bibitem{BFS07}
A.~Borodin, P.~Ferrari, T.~Sasamoto, {Transition between Airy$_1$ and Airy$_2$
  processes and TASEP fluctuations}, Comm. Pure Appl. Math. 61 (2008)
  1603--1629.

\bibitem{BFP06}
A.~Borodin, P.~Ferrari, M.~Pr{\"a}hofer, {Fluctuations in the discrete TASEP
  with periodic initial configurations and the Airy$_1$ process}, Int. Math.
  Res. Papers 2007 (2007) rpm002.

\bibitem{BF08}
A.~Borodin, P.~Ferrari, {Anisotropic Growth of Random Surfaces in $2+1$
  Dimensions}, Comm. Math. Phys. 325 (2014) 603--684.

\bibitem{BFP09}
J.~Baik, P.~Ferrari, S.~P{\'e}ch{\'e}, {Limit process of stationary TASEP near
  the characteristic line}, Comm. Pure Appl. Math. 63 (2010) 1017--1070.

\bibitem{Jo00b}
K.~Johansson, Shape fluctuations and random matrices, Comm. Math. Phys. 209
  (2000) 437--476.

\bibitem{BBP06}
J.~Baik, G.~{Ben Arous}, S.~P\'ech\'e, Phase transition of the largest
  eigenvalue for non-null complex sample covariance matrices, Ann. Probab. 33
  (2006) 1643--1697.

\end{thebibliography}

\end{document}